\label{key}\RequirePackage{fix-cm}
\documentclass[smallextended]{svjour3}       
\smartqed  

\usepackage[utf8]{inputenc}
\usepackage{microtype}
\usepackage{placeins}
\usepackage{floatrow}

\usepackage[low-sup]{subdepth}
\usepackage{mathtools}
\usepackage{graphicx}
\usepackage{mathptmx}      

\newcommand\xqed[1]{%
  \leavevmode\unskip\penalty9999 \hbox{}\nobreak\hfill
  \quad\hbox{#1}}
\newcommand\remarkend{\xqed{$\triangle$}}

\newcommand{\Oc}{\mathcal{O}}
\newcommand{\DO}{\Omega}   

\renewcommand{\epsilon}{\varepsilon} 

\renewcommand{\div}{\operatorname{div}\,}

\newtheorem{assumption}[lemma]{Assumption}

\usepackage{amsfonts,dsfont}
\usepackage{amsmath,amssymb} 

\usepackage{graphicx} 
\usepackage{subfigure} 
\usepackage{color}

\usepackage{multicol}

\usepackage{array}
\usepackage{booktabs}

\begin{document}

\title{Finite Element Error Estimates on Geometrically Perturbed Domains}

\author{Piotr Minakowski \and Thomas Richter}


\institute{
    P. Minakowski \at
    Institute of Analysis and Numerics, Otto-von-Guericke
    University Magdeburg, Universit\"atsplatz 2, 39106 Magdeburg,
    Germany,\\
    \email{  piotr.minakowski@ovgu.de}
    \and
    T. Richter \at
    Institute of Analysis and Numerics, Otto-von-Guericke
    University Magdeburg, Universit\"atsplatz 2, 39106 Magdeburg,
    Germany, 
    and Interdisciplinary Center for Scientific Computing, 
    Heidelberg University, INF 205, 69120 Heidelberg, Germany, \\
    \email{  thomas.richter@ovgu.de}
}

\date{Received: date / Accepted: date}

\maketitle

\begin{abstract}    
    We develop error estimates for the finite element approximation of
elliptic partial differential equations on perturbed domains,
i.e. when the computational domain does not match the 
real geometry. The result shows that the error related to the
domain can be a dominating factor in the finite element
discretization error. The main result consists of $H^1$- and
$L^2$-error estimates for the Laplace problem. Theoretical
considerations are validated by a computational example.
\keywords{perturbed domains,  finite elements, error estimates }
     \subclass{65N30 \and 65N15 \and 35J25}
\end{abstract}

\section{Introduction}

The main aim of this work is to develop finite element (FE)
error estimates in the case when there is uncertainty with
respect to the  computational domain. 	We consider the
question of how a domain related error affects the
finite element discretization error. We use the conforming finite
element method (FEM) which is well established in the scientific
computing community and allows for a rigorous analysis of the
approximation error~\cite{ErnGuermond2004}.

Our motivation is as follows. The steps to obtain a mesh for FE
computations often come with some uncertainty, for example related to
empirical measurements or image processing techniques, e.g. medical
image segmentation \cite{Oberkampf2006,OberkampfBook}. 
Therefore, we often perform computations on a domain
which is an approximation of the real geometry, i.e., the
computational domain is close to but does not match the real
domain. In this work we do not specify the source of the error, but we
take the error into account by explicitly using the error laden
reconstructed domains.

This theoretical result is of great importance for scientific computations. 
Vast numbers of engineering branches rely on the results of computational fluid 
dynamics  
simulations, where there is often uncertainty connected to the computational 
domain. 
A prime example of this is computational based medical diagnostics, where 
shapes are
reconstructed from inverse problems, such as computer tomography. The 
assessment of error 
attributed  to the limited spatial resolution of magnetic resonance techniques 
has been 
discussed in \cite{Moore1998,Moore1999}. For a survey on computational vascular 
fluid 
dynamics, where modeling and reconstruction related issues are discussed, we 
refer to 
\cite{Quarteroni2000}. Error analysis of computational models is a key factor 
for assessing the reliability for virtual predictions.

Uncertainties in the computational domain have been studied from the
numerical perspective. Rigorous bounds for elliptic problems on random
domains have been derived, for approximate problems defined on a
sequence of domains that is supposed to converge in the set sense to a
limit domain, for both Dirichlet \cite{Babuska2003} and Neumann
\cite{Babuska2002} boundary conditions. Although our techniques are
similar, we consider a case where the geometrical error is not small,
but where it might dominate the discretization error.

When measurement data is available the accuracy of numerical predictions can be improved 
by data assimilation techniques. Applications of variational data assimilation in 
computational hemodynamics have been reviewed in \cite{DElia2012}. For recent developments 
we refer to \cite{Funke2019} and \cite{Nolte}.
On the other hand, the treatment of boundary uncertainty can be cast
into a probabilistic framework. The domain mapping method is based entirely on
stochastic mappings to transform the original deterministic/stochastic
problem in a random domain into a stochastic problem in a
deterministic domain, see
\cite{Xiu2006,Tartakovsky2006,Harbrecht2016}. The perturbation method
starts with a prescribed perturbation field at the boundary of a
reference configuration and uses a shape Taylor expansion with respect
to this perturbation field to represent the solution
\cite{Harbrecht2008}. In \cite{Allaire2015} and \cite{Dambrine2015} a similar technique was used to incorporate random perturbations of a given domain in the context of shape optimization. Moreover, the fictitious domain approach and a
polynomial chaos expansion have been applied in \cite{Canuto2007}. We note,
that the probabilistic approach is beyond the scope of this work and the
introduction of the boundary uncertainty as random variable increases
the complexity of the problem.

The above approaches incorporate additional information on the domain reconstruction, 
such as measurement data or a probabilistic distribution of the approximation error.  
In comparison to these approaches our result can be seen as the worst case scenario. 
We only require that the distance between the two domains is bounded.

  The analysis presented in this paper starts with well-known results
  regarding the finite element approximation on domains with curved
  boundaries. But in contrast to these estimates we cannot expect the
  error coming from the approximation of the geometry is small or
  even converging to zero. Instead we split the error into a
  geometric approximation error between real domain and perturbed
  domain and into an error coming from the finite  element
  discretization of the problem on the perturbed domain. A
  central step is Lemma~\ref{theorem:uandur} which estimates the geometry
  perturbation. Having in mind that this error is not small and
  cannot be reduced by means of tuning the discretization, the typical
  application case is to balance both error contributions to
  efficiently reach the barrier of the geometry error. 
  Theorem~\ref{theorem:main} gives such optimally balanced 
  estimates that include both error contributions.

This paper is organized as follows. After this introduction, in
Section \ref{sec:math} we introduce the mathematical setting and
some required auxiliary results. Section \ref{sec:disc} covers
finite element discretization and proves the main results of this work.
We illustrate our result with
 computational examples in Section \ref{sec:comp}.

\section{Mathematical setting and auxiliary result}\label{sec:math}

\subsection{Notation}

Let $\DO\subset\mathds{R}^d$  be a domain with dimension
$d\in\{2,3\}$. By $L^2(\Omega)$ we denote the Lebesgue space of square
integrable functions equipped with the norm $\|\cdot\|_\Omega$. By
$H^1(\DO)$ we denote the space of $L^2(\DO)$ functions with
first weak derivative in $L^2(\Omega)$ and by $H^m(\DO)$ for
$m\in\mathds{N}_0$ we denote the corresponding generalizations with weak
derivatives up to degree $m\in\mathds{N}_0$. 
The norms in $H^m(\DO)$
are denoted by $\|\cdot\|_{H^m(\DO)}$. For convenience we use the
notation $H^0(\DO)\coloneqq L^2(\DO)$. 
By $H^1_0(\DO)$ we denote
the space of those $H^1(\DO)$ functions that have vanishing trace
on the domain's boundary $\partial\DO$ and we use the notation
$H^1_0(\DO;\Gamma)$ if the trace only vanishes on a part of the
boundary, $\Gamma\subset\partial\DO$. Further, by
$(\cdot,\cdot)_\DO$  we denote the $L^2(\DO)$-scalar product
and $\langle\cdot,\cdot\rangle_{\Gamma}$ the $L^2$-scalar 
product on a $d-1$ dimensional manifold $\Gamma$, 
e.g. $\Gamma=\partial\DO$. Moreover, $[\partial_n\psi]$ is the jump of
the normal derivative of $\psi$, i.e. for $x\in\Gamma$ with normal
$\vec n$ (that is normal to $\Gamma$)
$
[\partial_n \psi](x)
:=\lim_{h\searrow 0}\partial_n \psi(x+h\vec n)
- \lim_{h\searrow 0}\partial_n \psi(x-h\vec n). 
$

\subsection{Laplace equation and domain perturbation}

On  $\DO \subset\mathds{R}^d$ let $f\in L^2(\DO)$ be the given right
hand side. We consider the Laplace problem with  homogeneous Dirichlet 
boundary conditions, 
\begin{equation}\label{vessel:laplace} 
-\Delta u= 	f   \text{ in }\DO,\quad u=0\text{ on }\partial\DO.
\end{equation}
The variational formulation of this problem is given by: find $u\in
H^1_0(\DO)$, such that 
\begin{equation}\label{vessel:laplaceweak} 
 (\nabla u,\nabla\phi)_\DO =
  (f,\phi)_\DO\quad
  \forall \phi\in H^1_0(\DO).
\end{equation}
The
boundary $\partial\DO$ is supposed to have a 
parametrization in $C^{m+2}$, where $m\in\mathds{N}_0$. Given the
additional regularity $f\in H^{m}(\DO)$, $H^0(\DO)\coloneqq L^2(\DO)$, 
there exists a unique solution satisfying the a-priori estimate  
\begin{equation}\label{laplace:regularity}
  \|u\|_{H^{m+2}(\DO)} \le c \|f\|_{H^m(\DO)},
\end{equation}
see e.g.~\cite{evans}. 

In the following we assume that the \emph{real domain} $\DO$ is
not exactly known but only given up to an uncertainty. We
hence define a second domain, the \emph{reconstructed domain}
$\DO_r$ with a boundary that allows for $C^{m+2}$ parametrization.
The Hausdorff distance between both domains is then denoted by
$\Upsilon\in\mathds{R}$,
\[
     \Upsilon:=\operatorname{dist}(\partial\DO,\partial\DO_r)
     :=\max\{\adjustlimits\sup_{x\in \partial\DO}\inf_{y\in \partial\DO_r}|x-y|,\adjustlimits\sup_{y\in \partial\DO_r}\inf_{x\in \partial\DO}|x-y|\}.
\]  

\begin{figure}[t]
  \begin{center}
    \includegraphics[width=0.9\textwidth]{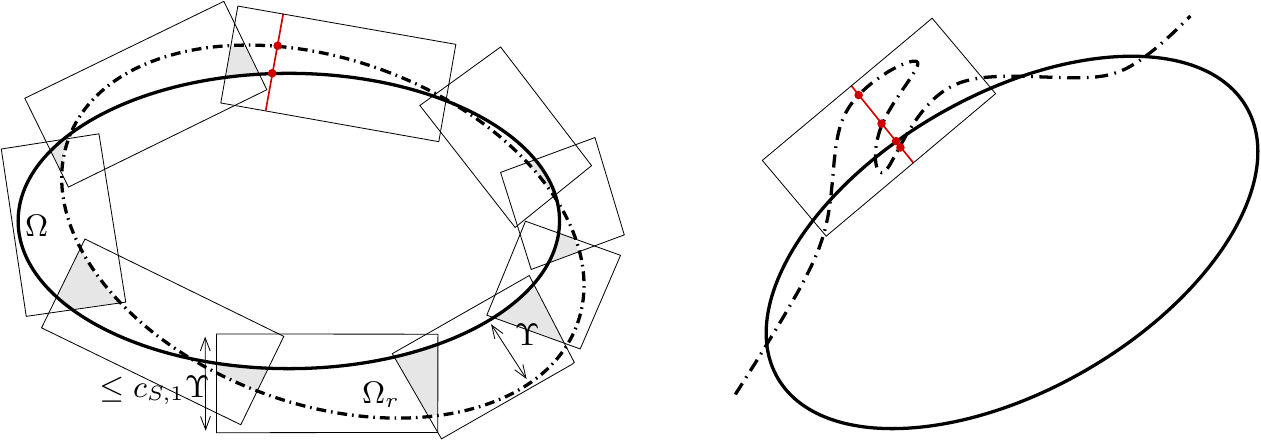}
    \caption{The domain $\DO$ (bold line) and its reconstruction
      $\DO_r$ (dashed). The cover of the domain remainders
      $S=(\Omega\setminus\Omega_r)\cup(\Omega_r\setminus \Omega)$ by a
      set of rectangles. The left configuration fulfils 
      Assumption~\ref{ass:domain}. The height $h_i$ of each rectangle $R_i$ is 
      bounded by $c_{S,1}\Upsilon$ and the intersections of each rectangle with
      the domain remainder $R_i\cap S$ do not overlap
      excessively. The shaded areas show the overlap.  The right configuration 
      is excluded by Assumption~\ref{ass:domain}.
    }
    \label{fig:cover}
  \end{center}
\end{figure}

This distance $\Upsilon$ is not necessarily small. When it comes to
spatial discretization we will be interested in both cases, $h\ll
\Upsilon$ as well as $\Upsilon\ll h$, where $h>0$ is the mesh size. 
The two domains do not match and either domain can protrude from the
other, see Figure~\ref{fig:cover}.
In order to prove our error estimates we require the following technical assumption on the
relation between the two domains $\Omega$ and $\Omega_r$. 

\begin{assumption}[Domains]\label{ass:domain}
  Let $\Omega$ and $\Omega_r$ be two domains with
  $\Omega\cap\Omega_r\neq \emptyset$ and with Hausdorff distance 
  $\Upsilon\in\mathds{R}$. Both boundaries allow for a local 
  $C^{m+2}$ parametrization, $m\in\mathds{N}_0$. Let
  \[
  S := (\Omega_r\setminus \Omega)\cup(\Omega\setminus\Omega_r).
  \]
  We assume that there exists a cover of $S$ by a finite number of
  open rectangles (or rectangular cuboids)
  $\{R_1,\dots,R_{n(S)}\}$. Each rectangle $R_i$ is given as translation
  and rotation of $(0,h_i)\times (0,t_i)$ for $d=2$, or $(0,h_i)\times
  (0,t_i^1)\times (0,t_i^2)$ for $d=3$, where the height $h_i$ is bounded by 
  $h_i\le c_{S,1}\, \Upsilon$ with a constant $c_{S,1}\ge
  0$. Following conditions hold: 
  \begin{enumerate}
  \item[\ref{ass:domain}.a)] On each rectangle $R$, the boundary lines
    $\partial\DO\cap R$ and $\partial\DO_r\cap R$ allow for unique
    parametrizations $g^R_\Omega(t)$ and $g^R_{\Omega_r}(t)$ over the
    base $t$, or $(t^1,t^2)$ for $d=3$, respectively. 
  \item[\ref{ass:domain}.b)] The area of the cover is bounded by the
    area of the remainder $S$, i.e.
    \[
    \big| \bigcup_{i=1}^{n(S)} R_i\cap S\big| \le c_{S,2} |S|,
    \]
    where $c_{S,2}>0$ is a constant.
  \end{enumerate}
  For the following we set $c_S:=\max\{c_{S,1},c_{S,2}\}$. 
\end{assumption}

Figure~\ref{fig:cover} shows such a cover for different domain
remainders. From Assumption A1 we deduce that each line through 
the height of the rectangle (marked in red in the figure) cuts each of the two
boundaries exactly one time. 
The second
assumption limits the overlap of the rectangles. These are shown as
the shaded in the left sketch in Figure~\ref{fig:cover}. 
Both assumptions on the domain 
are required for the proof of Lemma~\ref{lemma:help} that is based on Fubini's 
integral theorem. A more flexible framework that allows for a wider variety of
domains, e.g. with boundaries that feature hooks, could be based on
the construction of a map between two boundary segments on
$\partial\Omega_r$ and  $\partial\Omega$. Such approaches play an
important role in isogeometric analysis. We refer to  \cite{Xu2011}
and \cite{Xu2013} for examples on the construction of such maps. 

\medskip

To formulate the Laplace equation on the
reconstructed domain $\DO_r$ we must face the technical difficulty
that the right hand side $f\in H^m(\DO)$ is not necessarily defined
on $\DO_r$. We therefore weaken the assumptions on the
right hand side.
\begin{assumption}[Right hand side]\label{assumption:rhs}
  Let $f\in H^m_\text{loc}(\mathds{R}^d)$, i.e. $f\in H^m(G)$ for each compact
  subset $G\subset\mathds{R}^d$. 
  In addition we assume
  that the right hand side on $\Omega_r$ can be bounded by the right
  hand side on $\Omega$, i.e.
  \begin{equation}\label{bound:rhs}
    \|f\|_{H^m(\Omega_r)}\le c \|f\|_{H^m(\Omega)}. 
  \end{equation}
\end{assumption}
An alternative would be to use
  Sobolev extension theorems to extend functions $f\in H^m(\DO)$ from
  $\DO$ to $\DO_r$, see~\cite{Calderon1961}.

On $\DO_r$ we define the
solution $u_r\in H^1_0(\DO_r)$ to the perturbed Laplace problem
\begin{equation}\label{recon:laplace} 
  (\nabla u_r,\nabla\phi_r)_{\DO_r} = (f,\phi_r)_{\DO_r}
  \quad\forall \phi_r\in H^1_0(\DO_r),
\end{equation}
The unique solution to~(\ref{recon:laplace}) satisfies the bound 
\begin{equation}\label{laplace:regularity2}
  \|u_r\|_{H^{m+2}(\DO_r)} \le c \|f\|_{H^{m}(\DO_r)} \le c \|f\|_{H^{m}(\DO)}.
\end{equation}

\begin{remark}[Extension of the solutions]
  \label{remark:extendu}
  A difficulty for deriving error estimates is that $u$ is defined on $\DO$ and $u_r$ on~$\DO_r\neq \DO$. 
  Since the domains do not match, $u$ may
  not be defined on all of $\DO_r$ and vice versa. To give the
  expression $u-u_r$ a meaning on all domains we extend both solutions
  by zero outside their defining domains, i.e. $u:=0$ on
  $\DO_r\setminus \DO$  and $u_r:=0$ on $\DO\setminus\DO_r$. Globally,
  both functions still have the regularity $u,u_r\in
  H^1(\DO\cup\DO_r)$.  
  We will use the same notation for discrete functions $u_h\in V_h$ 
  defined on a mesh $\DO_h$ and extend them by zero to
  $\mathds{R}^d$.
  \remarkend
\end{remark}

The following preliminary results are necessary in the proof of the main estimates. 
They can be considered as variants of the trace inequality and of Poincar\'e's
estimate, respectively.
\begin{lemma}\label{lemma:help}
  Let $\gamma\in\mathds{R}, \gamma>0$, 
  $V\subset\mathds{R}^d$ and $W\subset\mathds{R}^d$ for $d\in
  \{2,3\}$ be two domains with boundaries $\partial V$ and $\partial
  W$  that satisfy Assumption~\ref{ass:domain} with distance
  \[
  \gamma\coloneqq \operatorname{dist}(\partial V,\partial W).
  \]
  For $\psi\in C^1(V )\cap C(\bar V)$ it holds
  \begin{equation}\label{help:1}
  \begin{split}
    &\|\psi\|_{\partial W\cap V}\le c \Big( \|\psi\|_{\partial V} + 
    \gamma^\frac{1}{2} \|\nabla \psi\|_{V\setminus W}\Big), \\
    &\|\psi\|_{V\setminus W} \le c\gamma^\frac{1}{2}
    \Big( \|\psi\|_{\partial V} + \gamma^\frac{1}{2}\|\nabla
    \psi\|_{V\setminus W}\Big),
  \end{split}
  \end{equation}
  where the constants $c>0$ depend on $c_S$ from
  Assumption~\ref{ass:domain} and the curvature of the domain
  boundaries. 
\end{lemma}
\begin{proof}
  Let $R$ be one rectangle of the cover and let $x_{\partial
    W}=g^R_{W}(t)\in \partial (W\cap V)\cap R$, see
  Figure~\ref{fig:lemma}. By $x_{\partial V}=g^R_{V}(t)\in \partial
  (W\cap V)\cap R$ we denote the corresponding  
  unique point on $\partial W\cap R$. The
  connecting line segment $\overline{x_{\partial V}x_{\partial W}}$
  completely runs through $V\setminus W$, as, if the line would leave
  this remainder, it would cut each line more than once which opposes
  Assumption \ref{ass:domain}.b). Integrating the function $\psi$ along this
  line gives 
  \[
  \big|\psi(x_{\partial W})\big|^2 \le 2\big|\psi(x_{\partial V})\big|^2 + 2\big| 
  \int_{x_{\partial V}}^{x_{\partial W}}\psi'(s)\,\text{d}s\big|^2. 
  \]
  Applying H\"older's inequality to the second term on the right hand side, 
  with the length of the line   bounded by $c_S\gamma$, we obtain
  \[
  |\psi(x_{\partial W})|^2 \le 2|\psi(x_{\partial V})|^2 + 2c_S\gamma
  \int_{x_{\partial V}}^{x_{\partial W}}|\nabla
  \psi(s)|^2\,\text{d}s. 
  \]
  Using the parametrizations $x_{\partial W}=g^R_{W}(t)$ and
  $x_{\partial V}=g^R_{V}(t)$ we integrate over $t$ which gives
  \begin{equation}\label{s1}
    \int |\psi(g^R_{W}(t))|^2\,\text{d}t
    \le \int |\psi(g^R_{V}(t))|^2\,\text{d}t
    +2c_S\gamma\int
    \int_{g^R_{V}(t)}^{g^R_{W}(t)}|\nabla\psi(s)|^2\,\text{d}s\,\text{d}t. 
  \end{equation}
  The volume integral on the right hand side is exactly the integral
  over $R\cap (V\setminus W)$. The boundary integrals can be
  interpreted as path integrals and therefore be estimated by
  \begin{multline}\label{s2}
    \frac{1}{\max_t \{1+|\nabla g^R_{W}(t)|^2\}} \int_{(\partial
      W\cap V)\cap R}
    |\psi|^2\,\text{d}s\\
    \le
    \frac{1}{\min_t \{1+|\nabla g^R_{V}(t)|^2\}} \int_{R\cap \partial V}
    |\psi|^2\,\text{d}s
    +2c_S\gamma \!\!\!\! \int_{R\cap (V\setminus W)}\!\!\!\! |\nabla\psi|^2\,\text{d}x. 
  \end{multline}
  As the boundaries allow for a $C^2$ parametrization, we
  estimate
  \begin{equation}\label{s3}
    \|\psi\|^2_{(\partial
      W\cap V)\cap R}\le c(\partial V,\partial W) c_S
    \Big(\|\psi\|^2_{R\cap\partial V} + \gamma \|\nabla \psi\|^2_{R\cap
      S}\Big). 
  \end{equation}
  Summation over all rectangles and estimation of all overlaps by
  means of Assumption~\ref{ass:domain} gives
  \[
  \|\psi\|^2_{\partial W\cap V}\le c(\partial V,\partial W) c_S
  \Big(\|\psi\|^2_{\partial V} + \gamma \|\nabla \psi\|^2_{V\setminus W}\Big). 
  \]
  For $\psi\in H^1_0(V)$, the term on $\partial V$ vanishes.
  
  \begin{figure}[t]
    \floatbox[{\capbeside\thisfloatsetup{capbesideposition={left,top},capbesidewidth=0.66\textwidth}}]{figure}[\FBwidth]
             {\caption{Illustration of the proof of
                 Lemma~\ref{lemma:help}. Points in each rectangle $R$ can be
                 represented by local coordinates $(t,s)$, where $0\le s\le
                 c_S\gamma$ and the range of $t$ depends on the size of the
                 rectangle. The two boundary segments are given by the (smooth)
                 parametrizations $g^R_V(t)$ and $g^R_{W}(t)$. Each line in
                 the direction of $s$ cuts both boundaries 
                 exactly once. In the 3d setting, the base is represented by
                 two coordinates $\mathbf{t}=(t_1,t_2)$. }
               \label{fig:lemma}}
             {\includegraphics[width=0.3\textwidth]{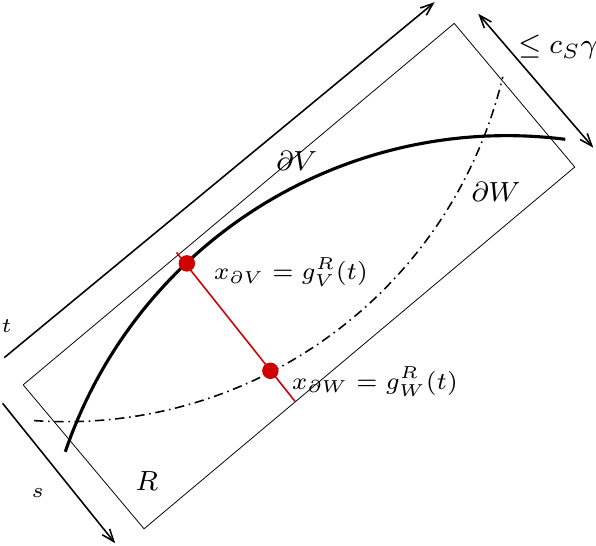}}
  \end{figure}
             
  To show the second estimate on $V\setminus W$ we again pick one
  rectangle $R$ and consider a point $x\in V\setminus W$ on the line
  connecting $x_{\partial V} = g^R_{V}(t)$ and $x_{\partial W} =
  g^R_{W}(t)$ such that we introduce the notation $x=x(t,s)$.
  By the same arguments as above it holds  
  \[
  |\psi(x(t,s))|^2 \le 2|\psi(g^R_V(t))|^2 + 2c_S \gamma 
  \int \int_{x(t)}^{g^R_V(t)}|\nabla
  \psi|^2\,\text{d}s\text{d}t.
  \]
  We integrate over $s$ and $t$ to obtain
  \[
  \|\psi(x)\|^2_{R\cap (V\setminus W)}\le
  \frac{2}{\min_t \{1+|\nabla g^R_{V}(t)|^2\}}
  \gamma \|\psi\|^2_{R\cap \partial V}
  + 2c_S \gamma^2
  \|\nabla\psi\|^2_{R\cap (V\setminus W)}. 
  \]
  Summing over all rectangles gives the desired result.
\end{proof}

The above lemma is later used in such a way that $V$ and $W$ can be 
substituted as~both $\Omega$ and $\Omega_r$, specifically to the case of use. 

\medskip

We continue by estimating the difference between the solutions of the
Laplace equations on $\DO$ and on $\DO_r$. 
\begin{lemma}\label{theorem:uandur}
  Let $\DO, \DO_r \in \mathds{R}^d$ with $\partial \DO,\partial \DO_r 
  \in C^{m+2}$  satisfying
  $\operatorname{dist}(\partial\DO,\partial\DO_r) < \Upsilon$ as~well
  as Assumption~\ref{ass:domain}. Furthermore, let $f\in
  L^2_\text{loc}(\mathds{R}^d)$ satisfy
  Assumption~\ref{assumption:rhs} and let $f_r\coloneqq f|_{\DO_r}$.
  For the solutions $u\in H^1_0(\DO)\cap H^2(\DO)$ and 
  $u_r\in H^1_0(\DO_r)\cap H^2(\DO_r)$
  to~(\ref{vessel:laplaceweak}) and~(\ref{recon:laplace}) 
  respectively, it holds that
  \[
  \|u-u_r\|_\DO + \Upsilon^\frac{1}{2}\|\nabla (u-u_r)\|_{\DO} \le c
  \Upsilon \|f\|_{\DO\cup\DO_r}.
  \]
\end{lemma}
\begin{proof}
  \emph{(i)}
  We continuously extend $u$ and $u_r$ by zero to $ \mathds{R}^d$,
  c.f. Remark~\ref{remark:extendu}, such that
  $u-u_r\in H^1(\DO\cup\DO_r)$ 
  is well defined. 
  We separate the domains of integration and integrate by parts  
  \begin{equation}\label{du:0x}
  \begin{split}
  \|\nabla(u-u_r)\|_\DO^2 =& \big(\nabla(u-u_r),\nabla(u-u_r)\big)_{\DO\cap\DO_r} +  \big(\nabla(u-u_r),\nabla(u-u_r)\big)_{\DO \setminus\DO_r}\\
  =&-\big(\Delta (u-u_r),u-u_r\big)_\DO + \langle
  \partial_n(u-u_r),u-u_r\rangle_{\partial(\DO \cap\DO_r)} \\
  &\quad +\langle \partial_n(u-u_r),u-u_r\rangle_{\partial(\DO \setminus\DO_r)}.
  \end{split}
 \end{equation}
 Combining the boundary terms on the right-hand side of \eqref{du:0x},
 into an integral over $\partial \Omega$ and a jump term over
 $\partial\DO_r\cap\DO$, we obtain
  \begin{equation}\label{du:0}
  \begin{split}
  \|\nabla(u-u_r)\|_\DO^2 =& -\big(\Delta (u-u_r),u-u_r\big)_\DO + \langle
  \partial_n(u-u_r),u-u_r\rangle_{\partial\DO}  \\
  &\quad+\langle [\partial_n (u-u_r)],u-u_r\rangle_{\partial\DO_r\cap\DO}.
  \end{split}
  \end{equation}

  In $\DO\cap \DO_r$ it holds $f=f_r$ and hence (weakly) $-\Delta
  (u-u_r)=0$, such that
  \begin{equation}\label{du:2}
    -(\Delta (u-u_r),u-u_r)_\DO
    =  -(\Delta (u-u_r),u-u_r)_{\DO\cap\DO_r}
     -(\Delta u,u)_{\DO\setminus\DO_r}
    \!\!=  (f,u)_{\DO\setminus\DO_r}.
  \end{equation}
  On $\partial\DO$ it holds $u=0$ and on $\partial\DO_r\cap \DO$ 
  it holds $u_r=0$. Further, since $u\in H^2(\DO)$ it holds that
  $[\partial_n u]=0$ on $\partial\DO_r\cap \DO$. Finally, 
  $u_r=0$ on  $\DO\setminus\DO_r$, such that the boundary 
  terms reduce to
  \begin{multline}\label{du:3}
    \langle \partial_n(u-u_r),u-u_r\rangle_{\partial\DO}
    +
    \langle [\partial_n(u-u_r)],u-u_r\rangle_{\partial\DO_r\cap \DO}\\
    =
    -\langle \partial_n(u-u_r),u_r\rangle_{\partial\DO\cap\DO_r}
    -\langle \partial_n u_r,u\rangle_{\partial\DO_r\cap \DO}. 
  \end{multline}
  Combining~(\ref{du:0})-(\ref{du:3}) 
  and using the Cauchy-Schwarz inequality, we estimate
  \begin{multline}\label{du:4}
    \|\nabla (u-u_r)\|_\DO^2 \le \|f\|_{\DO\setminus\DO_r}
    \|u\|_{\DO\setminus\DO_r}\\
    +\|\partial_n(u-u_r)\|_{\partial\DO\cap\DO_r}
    \|u_r\|_{\partial\DO\cap\DO_r}
    +\|\partial_nu\|_{\partial\DO_r\cap \DO}
    \|u\|_{\partial\DO_r\cap\DO}. 
  \end{multline}
  Since $u,u_r\in H^2(\DO\cap\DO_r)$, the trace inequality 
  gives
  \begin{multline}\label{du:5}
    \|\nabla (u-u_r)\|_\DO^2 \le
    \|f\|_{\DO\setminus\DO_r}
    \|u\|_{\DO\setminus\DO_r}\\
    +
    c\big(\|u\|_{H^2(\DO)}+\|u_r\|_{H^2(\DO_r)}\big)
    \big(
    \|u_r\|_{\partial\DO\cap\DO_r} + 
    \|u\|_{\partial\DO_r\cap\DO}\big).
  \end{multline}
  Applying Lemma \ref{lemma:help} twice: to $\psi=u$ and to
  $\psi=\nabla u$ (same for $u_r)$, and extending the norms from
  $\Omega\setminus\Omega_r$ to $\Omega$ and from
  $\Omega_r\setminus\Omega$ to $\Omega_r$ give the bounds
  \begin{equation}\label{du:5.5}
    \begin{aligned}
      \|u\|_{\partial\DO_r\cap \DO}&\le c \Upsilon^\frac{1}{2} \|\nabla
      u\|_{\DO\setminus\DO_r}
      \le c\Upsilon \Big( \|\nabla u\|_{\partial\DO} + \Upsilon^\frac{1}{2}
      \|u\|_{H^2(\DO)}\Big), \\
      \|u_r\|_{\partial\DO\cap \DO_r}&\le c\Upsilon^\frac{1}{2} \|\nabla
      u_r\|_{\DO_r\setminus\DO}\le c \Upsilon
      \Big( \|\nabla u_r\|_{\partial\DO_r} + \Upsilon^\frac{1}{2}
      \|u_r\|_{H^2(\DO_r)}\Big).
    \end{aligned}
  \end{equation}
  With the trace inequality and the a priori estimates
  $\|u\|_{H^2(\DO)}\le c \|f\|_{\DO}$ and $\|u_r\|_{H^2(\DO_r)} \le c
  \|f_r\|_{\DO_r} \le c \|f\|_{\DO}$ we obtain the bounds
  \begin{equation}\label{du:5.1}
    \|u\|_{\partial\DO_r\cap \DO}\le c \Upsilon\|f\|_{\DO},\quad
    \|u_r\|_{\partial\DO\cap \DO_r}\le c \Upsilon\|f\|_{\DO}.
  \end{equation}
  Using the fact that $u=0$ on $\partial\DO$ we apply \eqref{help:1} twice 
  and use the trace inequality to get the estimate
  \begin{equation}\label{du:5.2}
    \|u\|_{\DO\setminus\DO_r}\le c \Upsilon \|\nabla u\|_{\DO\setminus
      \DO_r}
    \le c \Upsilon^\frac{3}{2}\Big(\|u\|_{H^2(\DO)} +
    \Upsilon^\frac{1}{2} \|u\|_{H^2(\DO)}\Big)\le c 
    \Upsilon^\frac{3}{2} \|f\|_{\DO}.
  \end{equation}
  We can then estimate
  $\|f\|_{\DO\setminus\DO_r}\le \|f\|_{\DO}$ by extending to the
  complete domain. Combining~(\ref{du:5}) with~(\ref{du:5.1})
  and~(\ref{du:5.2}) we obtain the estimate
  \[
  \|\nabla (u-u_r)\|^2_{\DO}\le  c\Big(\Upsilon^\frac{3}{2}+
  \Upsilon \Big) \|f\|^2_{\DO},
  \]
  which concludes the $H^1$-norm bound.

  \noindent\emph{(ii)} For the $L^2$-estimate we introduce the adjoint
  problem
  \[
  z\in H^1_0(\DO):\quad -\Delta z = \frac{u-u_r}{\|u-u_r\|_\DO}\text{
    in }\DO,
  \]
  which allows for a unique solution 
  satisfying the a-priori bound $\|z\|_{H^2(\DO)}\le c_s$ with the
  stability constant $c_s<\infty$. Testing with $u-u_r$  and 
  integrating by parts twice gives
  \begin{equation*}
\begin{aligned}
  \|u-u_r\|_\DO = &
  - (z,\Delta (u-u_r))_{\DO}
  + \langle z,\partial_n(u-u_r)\rangle_{\partial\DO} \\
  &\quad + \langle z,[\partial_n(u-u_r)]\rangle_{\partial\DO_r\cap\DO}
  -\langle \partial_n  z,u-u_r\rangle_{\partial\DO}. 
    \end{aligned}
\end{equation*}
  It holds $z=0$ and $u=0$ on $\partial\DO$, $[\partial_n u]=0$ on
  $\partial\DO_r\cap \Omega$ and $-\Delta(u-u_r)=0$ in $\Omega\cap\Omega_r$ such that we
  get 
  \begin{equation*}
    \begin{aligned}
      \|u-u_r\|_\DO &=
      (z,f)_{\DO\setminus\DO_r} 
      -\langle z,\partial_n u_r\rangle_{\partial\DO_r\cap\DO}
      +\langle \partial_n  z,u_r\rangle_{\partial\DO}\\
      &\le \|z\|_{\DO\setminus\DO_r} \|f\|_{\DO\setminus\DO_r}
      + \|z\|_{\partial\DO_r\cap\DO} \|\partial_n
      u_r\|_{\partial\DO_r\cap\DO}
      + \|\partial_n z\|_{\partial\DO} \|u_r\|_{\partial\DO}. 
    \end{aligned}
  \end{equation*}
  The boundary terms $\|z\|_{\partial\DO_r\cap\DO}$ and
  $\|u_r\|_{\partial\DO}$ are estimated with Lemma~\ref{lemma:help},
   the normal derivatives by the trace inequality and the terms on
  $\DO\setminus\DO_r$ by \eqref{help:1}
  \[
  \|u-u_r\|_\DO
  \le c \Upsilon^\frac{3}{2} \|z\|_{H^2(\DO)} \|f\|_\DO + c\Upsilon
  \|z\|_{H^2(\DO)} \|u_r\|_{H^2(\DO_r)} + 
  c\Upsilon\|z\|_{H^2(\DO)}\|u_r\|_{H^2(\DO_r)}. 
  \]
  The $L^2$-norm estimate follows by using the bounds
  $\|u\|_{H^2(\DO)}\le c\|f\|_{\DO}$, $\|u_r\|_{H^2(\DO_r)}\le c
  \|f\|_\DO$ and $\|z\|_{H^2(\DO)}\le c$. 
\end{proof}

\begin{remark}
  The estimate $\|f\|_{\DO\setminus\DO_r}\le c \|f\|_{\DO}$ is not
  optimal. Further powers of $\Upsilon$ are easily generated at the
  cost of a higher right hand side regularity. Also,
  the estimate $\|\partial_n (u-u_r)\|\le c
  (\|u\|_{H^2(\DO)}+\|u_r\|_{H^2(\DO_r)})$ by Cauchy Schwarz and the
  trace inequality could be enhanced to produce powers of
  $\Upsilon$. The limiting term in~(\ref{du:0}) however is the
  boundary integral $|\langle \partial_n
  u_r,u\rangle_{\partial\DO_r\cap \DO}|=\Oc(\Upsilon^\frac{1}{2})$
  which is optimal in the $H^1$-estimate. In~Remark~\ref{remark:opt} and
  Corollary~\ref{cor} we present an estimate that focuses on the
  intersection $\DO\cap\DO_r$ only and that  allows us to improve the order to
  $\Oc(\Upsilon)$ in the $H^1$-case by avoiding exactly this boundary
  integral.
  \remarkend
\end{remark}

\section{Discretization}\label{sec:disc}

The starting point of a finite element discretization is the
mesh of the domain $\DO$. In~our setting we do not
mesh $\DO$ directly, because the domain $\DO$ is not
exactly known. Instead, we consider a mesh of the reconstructed domain
$\DO_r$.  

We partition $\DO_r$ into a parametric triangulation $\DO_h$, 
consisting of open elements $T\subset\mathds{R}^d$. Each
element $T\in\DO_h$ stems from a unique reference element $\hat T$
which is a simple geometric structure such as a triangle, quadrilateral or
tetrahedron. The numerical examples in Section~\ref{sec:comp} are based
on quadrilateral meshes. The map $T_T:\hat T\to T$ is a polynomial of
degree 
$r\in\mathds{N}$. We will consider iso-parametric finite element
spaces, that are based on polynomials of the same degree
$r\ge 1$. In the following we assume structural and shape regularity of the
mesh such that standard interpolation estimates 
\begin{equation}\label{interpolation}
  \begin{aligned}
    \|\nabla^k (u_r-I_h u_r)\|_T&\le c
    h^{r+1-k}\|u_r\|_{H^{r+1}(T)},\quad k=0,\dots,r\le m,\\
    \|\nabla^k (u_r-I_h u_r)\|_{\partial T}&\le c
    h^{r+\frac{1}{2}-k}\|u_r\|_{H^{r+1}(T)},\quad k=0,\dots,r\le m,
  \end{aligned}
\end{equation}
will hold for all elements $T\in\DO_h$, c.f.~\cite{Bernardi1989}. 
The discretization parameter $h$  represents the size of the
largest element in the mesh. See~\cite[Section 4.2.2]{Richter2017} for a detailed description. 

On the reference element $\hat T$ let $\hat P$ be a polynomial space
of degree $r$, e.g. 
\[
\hat P \;\widehat{=}\; Q^r := \operatorname{span}\{x_1^{\alpha_1}\cdots
x_d^{\alpha_d}\;:\; 
0\le \alpha_1,\dots,\alpha_d\le r\}
\]
on quadrilateral and hexahedral meshes. Then, the finite element space
$V_h^r$ on the mesh $\DO_h$ is defined as
\[
V_h^r=\{\phi_h\in C(\bar \DO_h)\;:\; \phi_h\circ T_T\in \hat P\text{ on
  every }T\in\Omega_h\}.
\]
This parametric finite element space does not exactly match the
domain $\DO_r$. Given an iso-parametric mapping of degree $r$ it holds
$\operatorname{dist}(\partial\DO_r,\partial\DO_h) = \Oc(h^{r+1})$ and
finite element approximation error and geometry approximation error
are balanced. 
  Iso-parametric finite elements for the approximation
  on domains with curved boundaries are well
  established~\cite{ErgatoudisIronsZienkiewicz1968}, 
  optimal interpolation and finite element error estimates 
  have been presented in~\cite[Section 4.4]{Ciarlet2002}.
  The case of higher order elements with optimal order energy norm
  estimates is covered in~\cite{Lenoir1986}.
  From~\cite[Theorem 4.37]{Richter2017} we cite the following
  approximation result for the iso-parametric approximation of the
  Laplace equation that also covers the $L^2$-error and which is
  formulated in a similar notation.
\begin{theorem}\label{theorem:iso}
  Let $m\in\mathds{N}_{0}$ and let $\Omega_r$ be a domain with a boundary
  that allows for a parametrization of degree $m+2$. Let $f_r\in
  H^{m}(\Omega_r)$ and $u_h\in V_h^r\cap H^1_0(\DO_h)$
  be the iso-parametric finite  
  element discretization of degree $1\le r\le m+1$ 
  \[
  (\nabla u_h,\nabla\phi_h)_{\DO_h}= (f_r,\phi_h)_{\DO_h}\quad\forall
  \phi_h\in V_h^{r}.
  \]
  It holds
  \[
  \|u_r-u_h\|_{H^1(\DO_r)} \le ch^r \|f_r\|_{H^{r-1}(\DO_r)},\quad
  \|u_r-u_h\|_{\DO_r} \le ch^{r+1} \|f_r\|_{H^{r-1}(\DO_r)}.
  \]
\end{theorem}

We formulated the error estimate on the domain $\DO_r$ although the
finite element functions are given on $\DO_h$ only. To give
Theorem~\ref{theorem:iso} meaning, we consider all functions extended by
zero as described in Remark~\ref{remark:extendu}. Combining these
preliminary results directly yields the a priori error estimates.

\begin{theorem}\label{theorem:main}
  Let $m\in\mathds{N}_{0}$,  $\Omega$ and $\Omega_r$ be domains
  with $C^{m+2}$ boundary, distance $\Upsilon$ and that satisfy
  Assumption~\ref{ass:domain}.
  Let $\DO_h$ be the iso-parametric mesh of
  $\DO_r$ with degree $1\le r\le m+1$ and let $f\in
  H^{r-1}_{loc}(\mathds{R}^d)$ satisfy
  Assumption~\ref{assumption:rhs}.  For the finite element error 
  between the fully discrete solution $u_h\in V_h^{r}$ 
  \[
  (\nabla u_h,\nabla\phi_h)_{\DO_h}= (f,\phi_h)_{\DO_h} \quad
  \forall \phi_h\in V_h^{r}
  \]
  and the \emph{true solution} $u\in H^1_0(\DO)\cap H^{m+2}(\DO)$ it holds
  $$ \|u-u_h\|_{H^1(\Omega)} \le c \big(\Upsilon^\frac{1}{2} +
  h^r \big)\|f\|_{H^{r-1}(\Omega)}, $$
  as well as 
  $$ \|u-u_h\|_{\Omega} \le c (\Upsilon + h^{r+1}) \|f\|_{H^{r-1}(\Omega)}. $$
\end{theorem}
\begin{proof}
  \emph{(i)} We start with the $H^1$ error. Inserting $\pm u_r$ and extending the finite element error $u_r-u_h$ from $\DO$  to $\Omega_r$, where a small remainder appears, we have 
  \begin{equation}\label{te:1}
    \|\nabla (u-u_h)\|_\DO^2 \le 2\Big(
    \|\nabla (u-u_r)\|_\DO^2 + \|\nabla (u_r-u_h)\|_{\DO_r}^2
    + \|\nabla (u_r-u_h)\|^2_{\DO\setminus\DO_r}\Big). 
  \end{equation}
  The first and the second term on the right hand side are estimated by 
  Lemma~\ref{theorem:uandur} and~Theorem~\ref{theorem:iso} and, since $u_r=0$ on $\DO\setminus\DO_r$, we obtain 
  \begin{equation}\label{te:1.1}
    \|\nabla (u-u_h)\|_\DO^2 \le c \Big(\Upsilon + h^{2r}\Big)
    \|f\|_{H^{r-1}(\DO)}^2 + 2 \|\nabla u_h\|^2_{\DO\setminus\DO_r}.
  \end{equation}
  We continue with the remainder $\nabla u_h$ on $\DO\setminus\DO_r$, which is non-zero on $\DO_h$ only
  \[
  \|\nabla u_h\|_{\DO\setminus\DO_r}^2 = 
  \|\nabla u_h\|^2_{(\DO_h\setminus\DO_r)\cap (\DO\setminus\DO_r)}. 
  \]
  This remaining stripe has the width
  \[
  \gamma_{h,\Upsilon}\coloneqq {\cal O}(\min\{h^{r+1},\Upsilon\}), 
  \]
  and we apply Lemma~\ref{lemma:help} to get
  \begin{equation}\label{te:2.1}
    \|\nabla u_h\|_{(\DO\setminus\DO_r)\cap (\DO_h\setminus
      \DO_r)}^2\le c \gamma_{h,\Upsilon}  \|\nabla u_h\|_{\partial \DO_r}^2
    +c \gamma_{h,\Upsilon}^2  \|\nabla^2 u_h\|_{\DO_h\setminus\DO_r}^2,
  \end{equation}
  where the second derivative  $\nabla^2 u_h$ is understood
  element wise. This term is extended to $\DO_h$ and with
  the inverse estimate and the a priori estimate for the discrete
  solution we obtain with $\gamma_{h,\Upsilon}^2={\cal O}(h^{2r+2})$ that 
  \begin{equation}\label{te:2.2}
    \gamma_{h,\Upsilon}^2  \|\nabla^2 u_h\|_{\DO_h\setminus\DO_r}^2
    \le c_{inv} \gamma_{h,\Upsilon}^2\ h^{-2}  \|\nabla u_h\|_{\DO_h}^2
    \le c_{inv} h^{2r} \|f\|_{\DO_h}^2\le c h^{2r} \|f\|^2_{\DO}. 
  \end{equation}
  To the first term on the right hand side of~(\ref{te:2.1}) we add $\pm u_r$ and $\pm I_h u_r$, the nodal interpolation of $u_r$ into the finite element space
  \begin{equation}\label{te:2.3}
    \gamma_{h,\Upsilon} \|\nabla u_h\|^2_{\partial\DO_r}\le
    c\gamma_{h,\Upsilon}\big(
    \|\nabla u_r\|^2_{\partial\DO_r}
    +\|\nabla (u_r-I_h u_r)\|^2_{\partial\DO_r}
    +\|\nabla (u_h-I_h u_r)\|^2_{\partial\DO_r}\big).
  \end{equation}
  Here, the first and last terms are estimated with the trace inequalities and, 
  in the case of the discrete term with the inverse inequality\footnote{We 
  refer to \cite[Chapter 1.4.3]{DiPietro2012} or 
  \cite{Barrenechea2017,Cangiani2018} for recent developments on the local 
  trace inequality and the inverse estimate on meshes with curved boundaries.},
  followed by adding $\pm u_r$ we get
  \begin{multline}\label{te:2.31}
    \gamma_{h,\Upsilon} \|\nabla u_h\|^2_{\partial\DO_r}\le
    c\Upsilon \|f\|_{L^2(\DO)}^2 
    + ch^{r+1} \|\nabla (u_r-I_h u_r)\|^2_{\partial\DO_r}\\
    + 
    ch^r \|\nabla (u_r-I_h u_r)\|^2_{\DO_r}
    + h^r \|\nabla (u_h-u_r)\|^2_{\DO_r}.
  \end{multline}
  We used both $\gamma_{h,\Upsilon}={\cal O}(h^{r+1})$ and $\gamma_{h,\Upsilon}={\cal O}(\Upsilon)$. 
  Then, collecting all terms in~(\ref{te:1.1})-(\ref{te:2.31})
  and using the interpolation estimates as well as Theorem~\ref{theorem:iso} we finally get 
  \begin{equation}
    \|\nabla (u-u_h)\|^2_\DO \le 
    c \big(\Upsilon + h^{2r}\big) \|f\|_{H^{r-1}(\DO)}^2
    +ch^{3r-1}\|f\|_{H^{r-1}(\DO)}^2,
  \end{equation}
  which shows the a priori estimate since $3r-1\le 2r$ for all $r\ge 1$.

%
%
%
  
  \medskip\noindent\emph{(ii)} For the $L^2$-error we proceed in the
  same way, but the remainder appearing in~(\ref{te:1}) does not carry
  any derivative, such that, instead of~(\ref{te:2.1}) the optimal order
  variant of Lemma~\ref{lemma:help} with integration to the boundary
  $\partial\DO_h$, where $u_h=0$, can be applied, i.e.
  \[
  \|u-u_h\|^2_\DO \le c \Big( \|u-u_r\|^2_\DO + \|u_r-u_h\|_{\DO_r}^2
  +\Upsilon^2 \|\nabla u_h\|^2_{\DO\setminus\DO_r}
  \Big). 
  \]
  The $L^2$-estimate directly follows with Lemma~\ref{theorem:uandur},
  Theorem~\ref{theorem:iso} and by the a priori estimate $\|\nabla
  u_h\|^2_{\DO\setminus\DO_r}\le c \|\nabla u_h\|^2_{\DO_h}\le
  c\|f\|^2_{\DO_r}$.   

\end{proof}

\begin{figure}[t]
  \begin{center}
    \begin{minipage}{0.35\textwidth}
      \includegraphics[width=\textwidth]{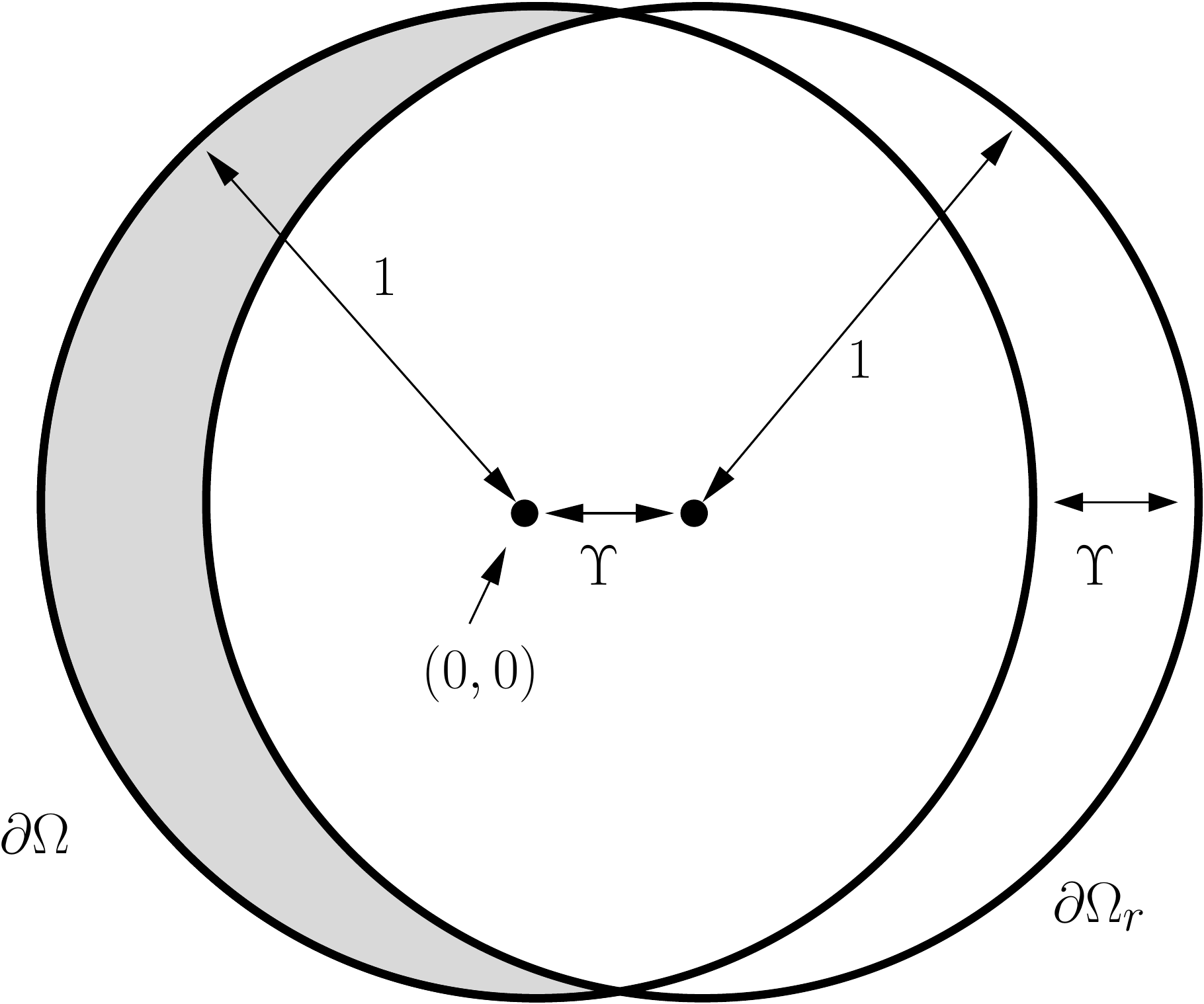}
    \end{minipage}\hspace{0.05\textwidth}
    \begin{minipage}{0.57\textwidth}\small
      On $\DO=B_1(0)$ and $\DO_r=B_1(\Upsilon)$ consider $-\Delta u=4$
      and $-\Delta u_r=4$, respectively with homogeneous Dirichlet
      conditions and the solutions
      \[
      u(x,y) = 1-x^2-y^2,\;
      u_r(x,y) = 1-(x-\Upsilon)^2-y^2
      \]
      and the errors
      \[
      \|\nabla (u-u_r)\|_{\DO} = \sqrt{8} \Upsilon^\frac{1}{2} +
      \Oc(\Upsilon),\;
      \|u-u_r\|_{\DO} = \sqrt{\pi} \Upsilon+
      \Oc(\Upsilon^3). 
      \]
    \end{minipage}
  \end{center}
  \caption{Illustration concerning Remark~\ref{remark:opt}. The error
    estimates for $u-u_h$ are optimal, if the error is evaluated on
    $\Omega$. The lowest order terms $\Oc(\Upsilon^\frac{1}{2})$
    appear in the shaded area $\DO\setminus\DO_r$ where $u_r$ and (most
    of) $u_h$ are zero. } 
  \label{fig:opt}
\end{figure}

  \begin{remark}[Polygonal domains]
    In two dimensions, the extension of the error estimates to the case of 
    convex polygonal domains, where  $u\in H^2(\DO)$ and $u_r\in H^2(\DO_r)$,
    is relatively straightforward.
    In this case, $\DO_h$ fits $\DO_r$ such that the finite element
    error $u_r-u_h$ can be estimated with the standard a priori result
    $\|u_r-u_h\|+h\|\nabla (u_r-u_h)\| \le c \|f\|$.
    The extension of Lemma~\ref{lemma:help}, which locally requires
    smoothness of the parametrizations $g^R_W(\cdot)$ and
    $g^R_V(\cdot)$, see steps~(\ref{s1})-(\ref{s3}), can be
    accomplished by refining the cover of the domain which is
    described in Assumption~\ref{ass:domain}, see also
    Figure~\ref{fig:cover}: 
    All rectangles are split in such a
    way that the corners of $\partial\DO$ and $\DO_r$ are cut by the
    edges of rectangles. This allows to derive the optimal error
    estimates $\|u-u_h\|_{H^1(\DO)}= {\cal O}(\Upsilon^\frac{1}{2}+h)$
    and $\|u-u_h\|_{H^1(\DO)}= {\cal O}(\Upsilon+h^2)$. In three
    dimensions, such a simple refinement of the cover is not possible and the 
    extension to polygonal domains is more involved.
    \remarkend
\end{remark}

\begin{remark}[Optimality of the estimates]\label{remark:opt}
  Two ingredients govern the error estimates:
  \begin{enumerate}
  \item
    A geometrical error of order $\Oc(\Upsilon^\frac{1}{2})$
    and $\Oc(\Upsilon)$, that describes the discrepancy
    between $\DO$ and $\DO_r$, in the $H^1$ and $L^2$ norms respectively.
    This term is optimal which is easily 
    understood by considering a simple example illustrated in
    Figure~\ref{fig:opt}, namely
    $-\Delta u=4$ on the unit disc $\DO=B_1(0)$ and $-\Delta u_r=4$
    on the shifted domain $\DO_r = B_1(\Upsilon)$. 
    The errors in $H^1$ norm and
    $L^2$ norms expressed on the complete domain $\DO$ are estimated
    by 
    \[
    \|u-u_r\|_\DO=\sqrt{\pi}\Upsilon + \Oc(\Upsilon^3),\quad
    \|\nabla(u-u_r)\|_\DO=\sqrt{8\Upsilon} + \Oc(\Upsilon).
    \]
    A closer analysis shows that the main error -- in the $H^1$-case
    -- occurs on the small shaded stripe $\DO\setminus\DO_r$ such that
    \[
    \|\nabla (u-u_r)\|_{\DO\setminus\DO_r} =
    \Oc(\Upsilon^\frac{1}{2}),\quad
    \|\nabla (u-u_r)\|_{\DO\cap\DO_r} = \Oc(\Upsilon),
    \]
    while the $L^2$-error in $\DO\cap\DO_r$ is optimal
    \[
    \|u-u_r\|_{\DO\setminus\DO_r} =
    \Oc(\Upsilon^\frac{3}{2}),\quad
    \|u-u_r\|_{\DO\cap\DO_r} = \Oc(\Upsilon).
    \]
  \item The usual Galerkin error $\|u_r-u_h\|_{\DO_r} +
    h\|\nabla (u-u_r)\|_{\DO_r} = \Oc(h^{r+1})$ of iso-parametric
    finite element approximations contributes to the overall
    error. For $\DO=\DO_r$, i.e. $\Upsilon=0$, this would be the
    complete error. This estimate is optimal, as it shows the same
    order as usual finite element bounds on meshes that resolve the
    geometry. 
  \end{enumerate}
  \remarkend
\end{remark}

In Section~\ref{sec:comp} we discuss the difficulty of measuring
errors on an unknown domain $\DO$. The optimality of the error
estimates is difficult to verify which is mainly due to the technical
problems in evaluating norms on the domain remainders
$\DO\setminus\DO_r$, where no finite element mesh is given. These
remainders contribute the lowest order parts $\Upsilon^\frac{1}{2}$ 
in the overall error. The following corollary is closer to the setting
of the numerical examples and it yields the approximation of order
$\Upsilon$ in the $H^1$-norm error. In addition to the previous
setting we require a regular map $T_r:\DO\to \DO_r$ between the two
domains. By pulling back $\DO_r$ to $\DO$ via this map a Jacobian
arises that controls the geometrical error and that hence has to be
controllable by $\Upsilon$. 

\begin{corollary}\label{cor}
  In addition to the assumptions of Theorem~\ref{theorem:main} let
  there be a $C^1$-diffeomorphism
  \[
  T_r:\DO\to\DO_r
  \]
  satisfying
  \begin{equation}\label{diff}
    \|I-\operatorname{det}(\nabla T_r) \nabla T_r^{-1}\nabla
    T_r^{-T}\|_{L^\infty(\DO)} =\Oc(\Upsilon).
  \end{equation}
  Further, let the following regularity of problem data hold in
  addition to Assumption~\ref{assumption:rhs}
  \begin{equation}\label{cor:regularitydata}
    f\in W^{1,\infty}_{loc}(\mathds{R}^d)\cap H^{r-1}_{loc}(\mathds{R}^d)
  \end{equation}
  and let the solution satisfy
  \begin{equation}\label{cor:regularity}
    \|u\|_{W^{2,\infty}(\DO)} +
    \|u_r\|_{W^{2,\infty}(\DO_r)} \le c. 
  \end{equation}
  Then, it holds 
  \[
  \|\nabla (u-u_h)\|_{\DO\cap\DO_r\cap\DO_h}\le c \big( \Upsilon + h^r \big).
  \]
\end{corollary}
\begin{proof}
  We start by splitting the error into
  domain approximation and finite element approximation errors
  \begin{equation}\label{cor:1}
    \|\nabla (u-u_h)\|_{\DO\cap\DO_r\cap\DO_h}\le 
    \|\nabla (u-u_r)\|_{\DO\cap\DO_r}+
    \|\nabla (u_r-u_h)\|_{\DO_r\cap\DO_h}.
  \end{equation}
  An optimal order estimate of the finite element error
  \begin{equation}\label{cor:1.5}
    \|\nabla  (u_r-u_h)\|_{\DO_r\cap\DO_h}\le 
    \|\nabla  (u_r-u_h)\|_{\DO_r}=\Oc(h^r)
  \end{equation}
  is given in Theorem~\ref{theorem:iso}.
  To estimate the first term of the right hand side of~(\ref{cor:1}) we~introduce the function
  \[
  \hat u_r(x):=u_r(T_r(x)),
  \]
  which satisfies $\hat u_r\in H^1_0(\DO)$ and solves the problem
  \[
  (J_r F_r^{-1}F_r^{-T}\nabla \hat u_r,\nabla\hat\phi_r)_{\DO} =
  (\hat f_r,\hat\phi_r)\quad\forall \hat\phi_r\in H^1_0(\DO),
  \]
  where $\hat f_r(x):=f(T_r(x))$ and where $F_r:=\nabla T_r$ and
  $J_r:=\operatorname{det}(F_r)$. See~\cite[Section
    2.1.2]{Richter2017}  for details of this transformation of the variational
  formulation. To estimate the domain approximation error
  in~(\ref{cor:1}) we introduce $\pm \hat u_r$ to obtain
  \begin{equation}\label{cor:2}
    \|\nabla (u-u_r)\|_{\DO\cap\DO_r}\le
    \|\nabla (u-\hat u_r)\|_{\DO\cap\DO_r}+
    \|\nabla (\hat u_r-u_r)\|_{\DO\cap\DO_r}.
  \end{equation}
  We introduce the notation $e_r:=u-\hat u_r$, extend the first term
  from ${\DO\cap\DO_r}$ to $\DO$  and insert 
  $\pm J_rF_r^{-1}F_r^{-T}\nabla\hat u_r$ which gives
  \begin{multline}\label{cor:3}
    \|\nabla (u-\hat u_r)\|_{\DO\cap\DO_r}^2\le 
    \|\nabla (u-\hat u_r)\|_{\DO}^2
    \\=
    (\nabla u,\nabla e_r)_{\DO} -
    (J_rF_r^{-1}F_r^{-T}\nabla \hat u_r,\nabla  e_r)_{\DO}
    +
    (J_rF_r^{-1}F_r^{-T}\nabla \hat u_r,\nabla  e_r)_{\DO}
    -(\nabla \hat u_r,\nabla e_r)_{\DO}\\
    =(f-\hat f_r,e_r)_{\DO} + ( [J_rF_r^{-1}F_r^{-T}-I]\nabla\hat
    u_r,\nabla e_r)_{\DO}\\
    \le \|f-\hat f_r\|_{\DO} c \|\nabla e_r\|_\DO
    + \|[J_rF_r^{-1}F_r^{-T}-I]\|_{L^\infty(\DO)} \|\nabla e_r\|_\DO,
  \end{multline}
  where we also used Poincar\'e's estimate. For bounding 
  $f-\hat f_r$ we consider a point $x\in \DO\cap \DO_r$, use the
  higher regularity of the right hand side~(\ref{cor:regularitydata})
  to estimate by a Taylor expansion 
  \begin{equation}\label{cor:33}
    |f(x)-\hat f_r(x)| = |f(x)-f(T_r(x))|=
    |\nabla f(\xi)\cdot(T_r(x)-x)|
    \le \Upsilon |\nabla f(\xi)|,
  \end{equation}
  where $\xi\in \DO$ is some point on the line from $x$ to
  $T_r(x)$. 
  We take the square and integrate over $\DO$ to get the
  estimate 
  \begin{equation}\label{cor:4}
    \|f-\hat f_r\|_\DO \le c \Upsilon
    \|f\|_{W^{1,\infty}(\Omega_\Upsilon)}, 
  \end{equation}
  where $\Omega_\Upsilon$ is a enlargement of $\DO$ by at most ${\cal
    O}(\Upsilon)$, since intermediate values $\xi$ used
  in~(\ref{cor:33}) are not necessarily part of $\DO\cup\DO_r$. 
  This argument is also applicable to the second term on the right hand side of~(\ref{cor:2})
  such that it holds 
  \[
  \|\nabla (\hat u_r-u_r)\|_M \le c\Upsilon
  \|u_r\|_{W^{2,\infty}(\DO\cap\DO_r)} \le c \Upsilon. 
  \]
  Combining this with~(\ref{cor:1}),~(\ref{cor:1.5}),~(\ref{cor:2})
  and (\ref{cor:3}) 
  finishes the proof. 
\end{proof}

Unfortunately this corollary can not be applied universally  
as the existence of a suitable map
$T_r:\DO\to\DO_r$ depends on the given application.
Here a construction, corresponding  to the ALE map, can be realised by
means of a \emph{domain deformation} $\hat d:\DO\to\mathds{R}^2$
\[
T_r(x) = x+\hat d(x),\quad F_r(x) = I + \nabla\hat d(x).
\]
Such a construction is common in fluid-structure interactions,
see~\cite[Section 2.5.2]{Richter2017}.
Given that $|\hat d|,|\nabla \hat d| = \Oc(\Upsilon)$ it holds
\[
\|J_r\|_{L^\infty(\DO)} = 1+\Oc(\Upsilon),\quad
\|I-J_rF_r^{-1}F_r^{-T}\|_{L^\infty(\DO)} = \Oc(\Upsilon). 
\]
While the assumption $|\hat d|=\Oc(\Upsilon)$ is easy to satisfy since
$\operatorname{dist}(\partial\DO,\partial\DO_r)\le \Upsilon$, the
condition 
$|\nabla \hat d|=\Oc(\Upsilon)$ will strongly depend on the shape and
regularity of the boundary.

\begin{figure}[t]
  \begin{center}
    \includegraphics[width=0.8\textwidth]{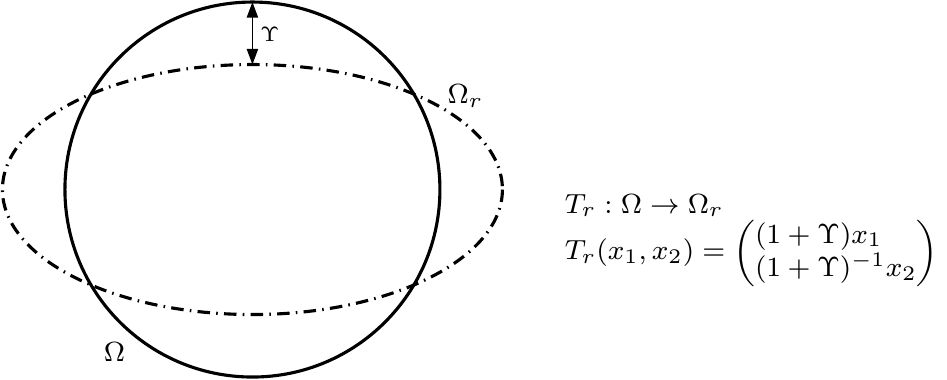}
  \end{center}
  \caption{Illustration of an example for the application of
    Corollary~\ref{cor}.} 
  \label{fig:cor}
\end{figure}

We conclude by discussing a simple application of this corollary.
Figure~\ref{fig:cor}  illustrates the setting. 
Let
$\Omega$ be the unit sphere, $\Omega_r$ be an ellipse
\[
\DO = \{x\in \mathds{R}^2\,:\, x_1^2+x_2^2<1\},\quad
\DO_r = \{x\in \mathds{R}^2\,:\,
(1+\Upsilon)^2x_1^2+(1+\Upsilon)^{-2}x_2^2<1\}. 
\]
It holds $\operatorname{dist}(\partial\DO,\partial\DO_r)\le
\Upsilon$ and we define the map $T_r:\DO\to\DO_r$ by 
\[
T_r(x) = \begin{pmatrix}
  (1+\Upsilon)^{-1} x_1 \\ (1+\Upsilon) x_2,
\end{pmatrix},\quad
F_r=\nabla T_r = \begin{pmatrix}
  (1+\Upsilon)^{-1}&0\\0&(1+\Upsilon)
\end{pmatrix},\quad J_r=1. 
\]
This map satisfies the assumptions of the corollary
\[
I-J_r F_r^{-1}F_r^{-T} =\Upsilon(\Upsilon+2)
\begin{pmatrix}
  -1&0\\0& (1+\Upsilon)^{-2}
\end{pmatrix},\quad
\|I-J_r F_r^{-1}F_r^{-T}\|_{\infty}=2\Upsilon+\Upsilon^2. 
\]

\section{Numerical illustration}\label{sec:comp}

\begin{figure}[t]
  {\includegraphics[width=.44\textwidth]{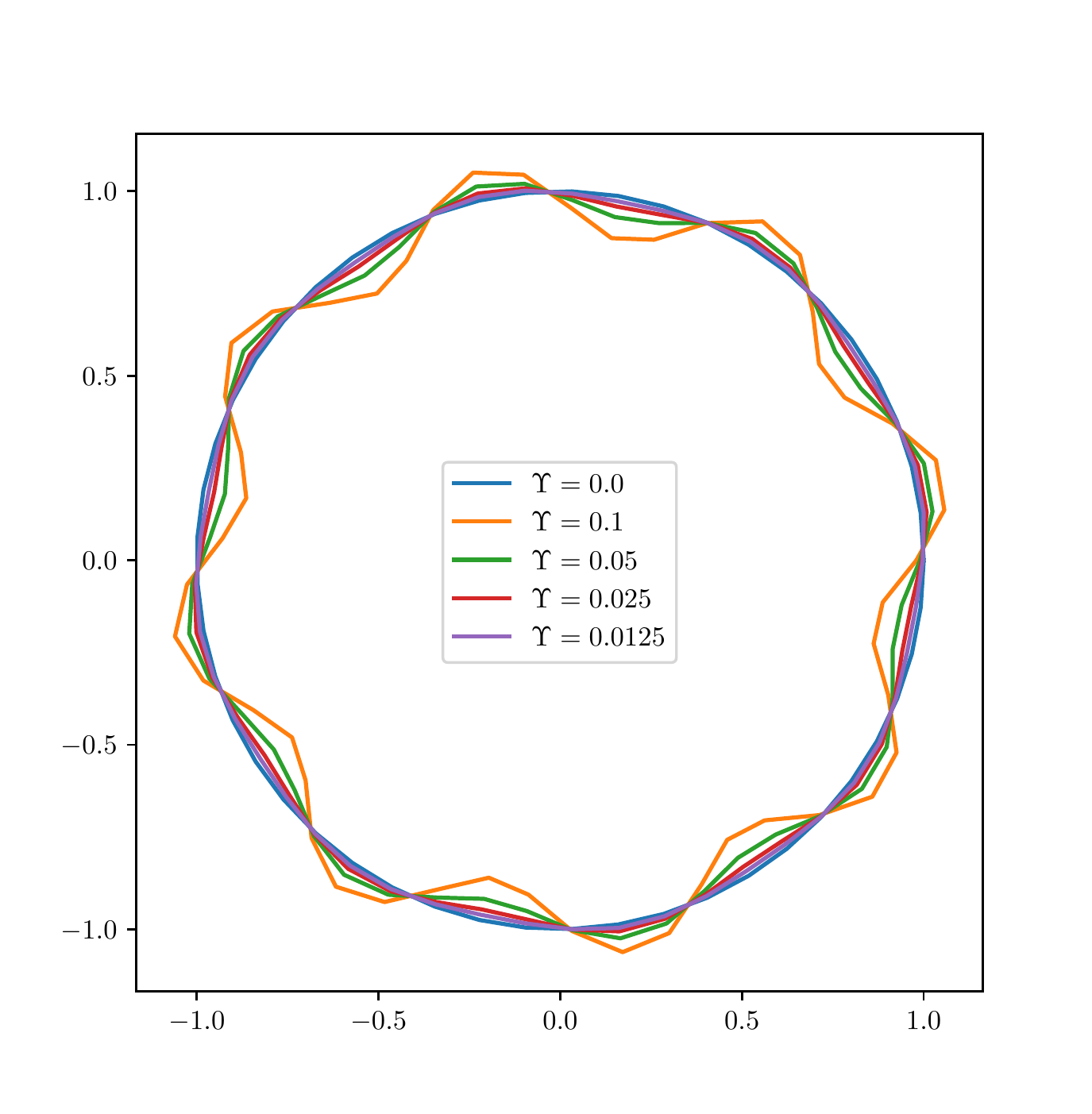}}
  {\includegraphics[width=.55\textwidth]{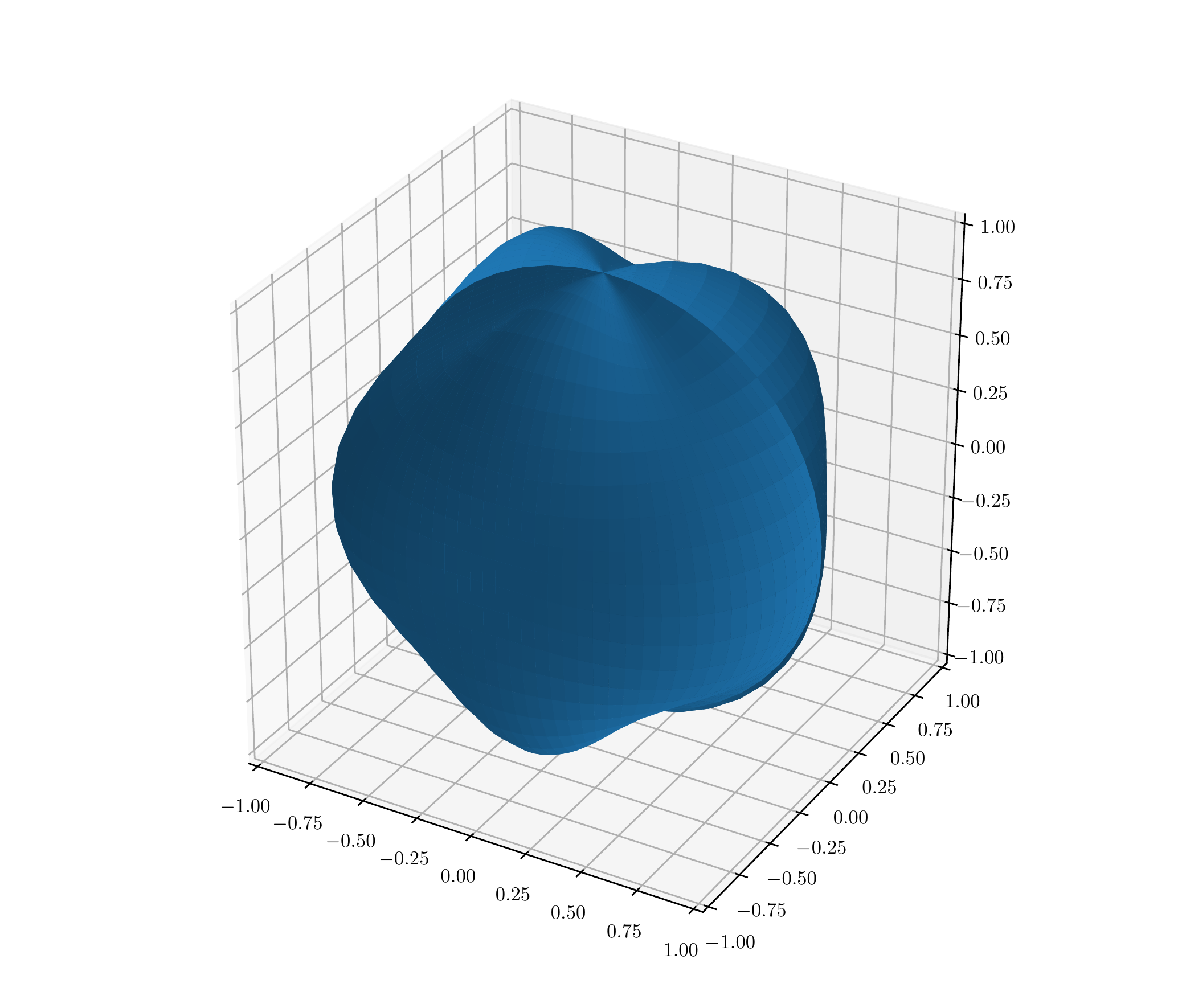}}
  \caption{Sketch of the computational domains w.r.t. the parameter $\Upsilon$ 
  in two dimensions (left) and for $\Upsilon=0.1$ in three dimensions (right).}
  \label{fig:domain}
\end{figure}

In this section we illustrate the theoretical results from the previous section. We compute the Laplace problem on a family of domains representing different values of $\Upsilon$. Moreover, we numerically extend the analytical predictions and show that a similar behavior holds for the Stokes system.

We consider $\DO$ to be a unit ball in two and three dimensions and
define a family of perturbed domains $\DO_\Upsilon$, with the
amplitude of the perturbation being dependent on the coefficient
$\Upsilon$, cf.~Figure \ref{fig:domain}.

In two dimensions, the boundary of the domain $\DO_\Upsilon$ 
is given in polar coordinates~$(\rho,\varphi)$ by
\[
\partial\DO_\Upsilon = \{(1 -
\Upsilon/5+\Upsilon\sin(8\varphi),\varphi) \text{  for  } \varphi \in
[0,2\pi)\}, 
\]
and in three dimensions in spherical coordinates $(\rho,\theta,\varphi)$ by
\[
\partial\DO_\Upsilon = \{(1 - \Upsilon/5+\Upsilon \sin(3\varphi)\sin(3\theta),\theta,\varphi) \text{  for  } \theta \in [0,\pi), \varphi \in [0,2\pi)\}.
\]
For computations we take
\[
\Upsilon \in \{0,0.0125,0.025,0.05,0.1 \}.
\]

In order to illustrate the convergence result from Theorem~\ref{theorem:main}, 
we compute the model problem on a series of uniformly refined meshes.
The dependence between the mesh size $h$ and the refinement level $L$
reads $h = 2^{-L}$. We denote the mesh approximating $\DO_\Upsilon$, with a mesh size $h$,
by $\DO_{h,\Upsilon}$.

The numerical implementation is realized in the software library Gascoigne 3D~\cite{Gascoigne3d},
using iso-parametric finite elements of degree $1$ and $2$. A detailed description of the underlying numerical 
methods is given in~\cite{Richter2017}.

\begin{figure}[t]
  \begin{center}
    \includegraphics[width=.65\textwidth]{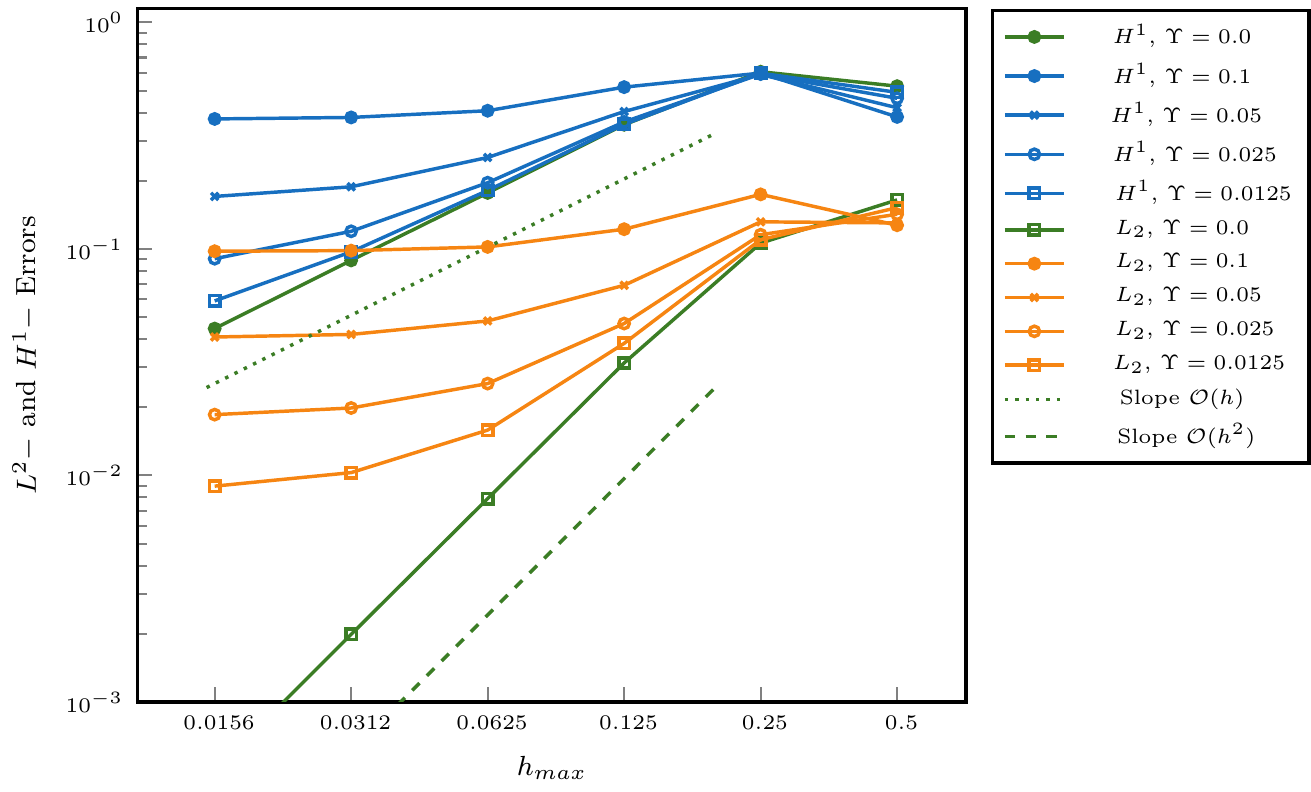}
  \end{center}	
  \caption{$L^2$- and $H^1$-errors w.r.t. mesh-size $h_{max}$ for varying parameter $\Upsilon$ computed for the Laplace problem in three-dimensions with linear finite elements.}
  \label{fig:3dq1}
\end{figure}

\subsection{Laplace equation in two and three dimensions}

We consider the following problem
\begin{equation}\label{stronglaplace}
  -\Delta u = 	f   \text{ in }\DO,\quad 
  u = 0 \text{ on }\partial\DO,
\end{equation}
where $\DO$ is the unit ball in two dimensions and the unit sphere in three
dimensions. 

To compute errors we choose a rotationally symmetric analytical
solution to \eqref{stronglaplace} as
\[
u(r)= - \cos\left(\frac{\pi}{2} r\right)
\]
with $r=\sqrt{x^2+y^2}$ in two and $r=\sqrt{x^2+y^2+z^2}$ in three
dimensions, respectively, 
which results in the right hand sides
\begin{align*}
f_{2d}(r) = \frac{\pi}{2r} \sin\left(\frac{\pi}{2} r\right) +
\frac{\pi^2}{4} \cos\left(\frac{\pi}{2} r\right) ,\quad f_{3d}(r) =
\frac{\pi}{r} \sin\left(\frac{\pi}{2} r\right) + \frac{\pi^2}{4}
\cos\left(\frac{\pi}{2} r\right). 
\end{align*}

\begin{figure}[t]
  \begin{center}
    \includegraphics[width=\textwidth]{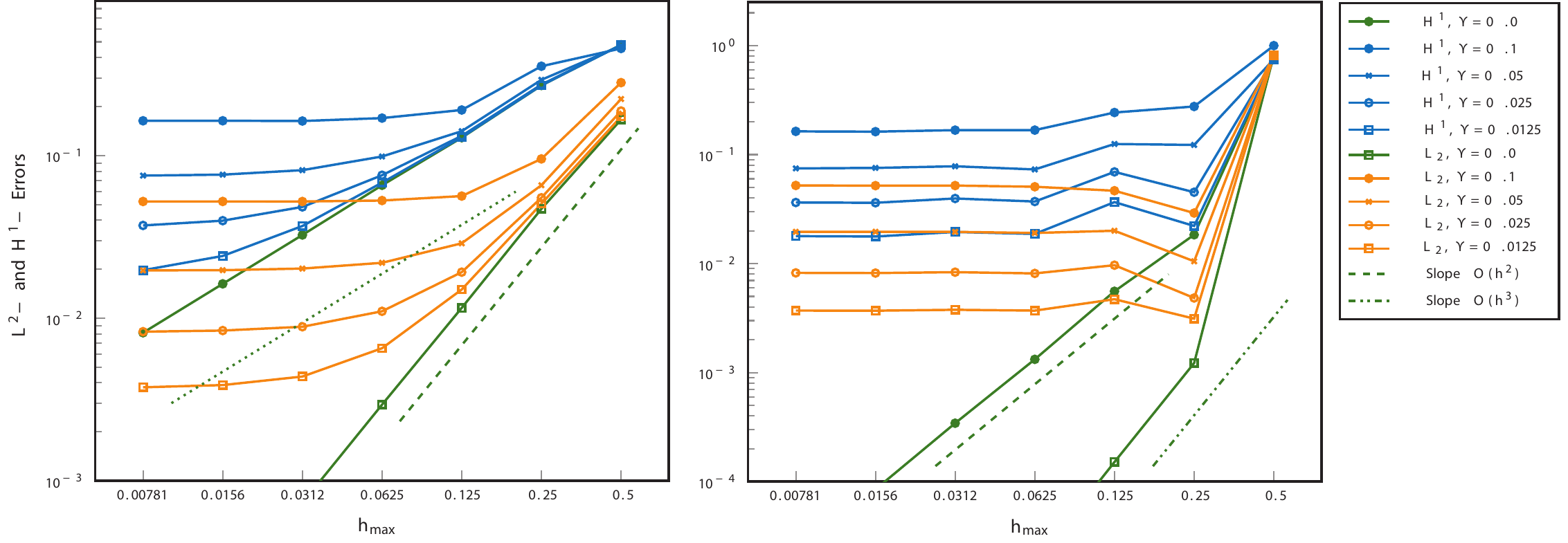}
  \end{center}	
  \caption{$L^2$- and $H^1$-errors w.r.t. mesh-size $h_{max}$ for
    varying parameter $\Upsilon$ computed for the Laplace problem in
    two-dimensions with FE. Left: linear finite elements. Right:
    quadratic finite elements.}
  \label{fig:2dq1q2}
\end{figure}

For the ease of evaluations the errors, the $H^1$- and $L^2$-norms will be computed on the truncated domains  
\[
\begin{aligned}
\DO'_{2d} &=\{(\varphi,\rho) \text{ for } \varphi \in [0,2\pi)  \text{ and } \rho \in (0,0.88) \},\\
\DO'_{3d} &= \{(\varphi,\theta, \rho) \text{ for }\theta \in [0,\pi),
  \varphi \in [0,2\pi) \text{ and } \rho \in (0,0.88) \},
\end{aligned}
\]
see also Remark~\ref{remark:opt}.
We hence do not compute the errors $\|\nabla (u-u_h)\|$ and
$\|u-u_h\|$ on the remainders $\DO\setminus\DO_r$. Therefore we expect
optimal order convergence in the spirit of Corollary~\ref{cor}. The
restriction of the domain to an area within $\DO_h$ is also by
technical reasons, as the evaluation of integrals outside of the meshed area is not easily possible.

\begin{figure}[t]
  \begin{center}
    \includegraphics[width=.65\textwidth]{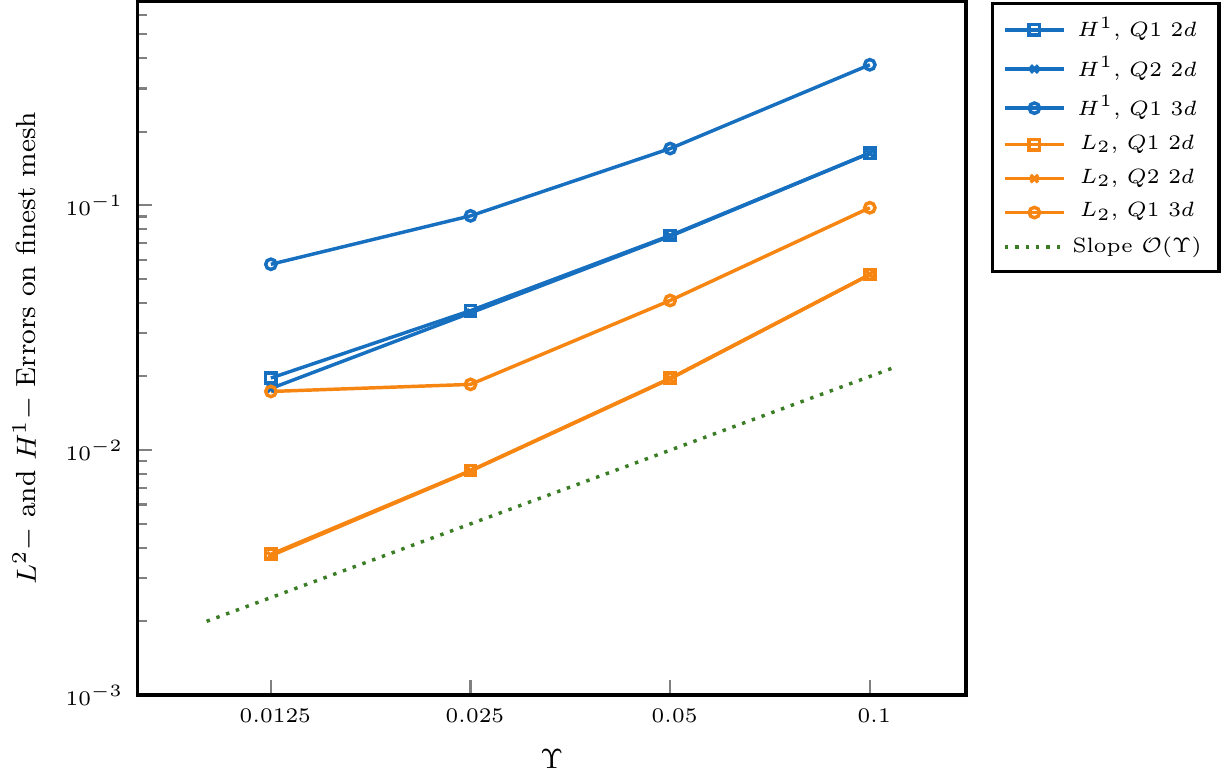}
  \end{center}	
  \caption{$L^2$- and $H^1$-errors w.r.t. parameter $\Upsilon$ computed for the Laplace problem in two and three-dimensions with linear and quadratic finite elements.}
  \label{fig:upsilon}
\end{figure}

In Figures \ref{fig:3dq1} and \ref{fig:2dq1q2} we see the resulting $L^2$- and
$H^1$-errors. We observe that for finer meshes, $\Upsilon$ becomes the
dominating factor of the error. In particular the use of quadratic
finite elements shows a strong imbalance between FE error and geometric
error, which quickly dominates as seen in the left part of
Fig.~\ref{fig:2dq1q2}. The 
result is consistent with 
Corollary~\ref{cor}.  As soon as the FE error is smaller than the
geometry perturbation $\Upsilon$, we do not observe any further
improvement of the error. In Fig.~\ref{fig:upsilon} we show the
convergence in both norms in terms of the geometry parameter
$\Upsilon$. Linear convergence is clearly observed. The apparent decay of
convergence rate in case of the $L^2$-error in three dimensions is
due to the still dominating FE error in this case.

\subsection{Stokes system in two dimensions}

To go beyond the Laplace problem, we investigate the behavior of the solution to the Stokes system with respect to 
the domain variation in two spatial dimensions. The problem is to find the velocity $\textbf{u}$ and the pressure 
$p$ such that
\begin{equation}\label{stokes}
  \div\textbf{u} = 0,\quad 
  -\Delta \textbf{u} +\nabla p = \textbf{f}\text{ in }\DO,
\end{equation}
with homogeneous Dirichlet condition $\textbf{u}=0$ on the boundary
$\partial\DO$ and a right hand side vector $\textbf{f}$. 
System \eqref{stokes} is solved with equal-order iso-parametric finite
elements using pressure stabilization by local projections,
see~\cite{BeckerBraack2001}. 

We prescribe an analytical solution for comparison with the finite
element approximation
\[
\textbf{u}(x,y) =
\cos\left(\frac{\pi}{2}(x^2+y^2)\right)\begin{pmatrix}y \\ -x
\end{pmatrix},
\]
where the corresponding forcing term reads
\[
\textbf{f}(x,y) =\pi \cos\left(\frac{\pi}{2}(x^2+y^2)\right)
\begin{pmatrix} yr^2 \pi 
+ 
4(y-x)\tan\left(\frac{\pi}{2}(x^2+y^2)\right)\\
-xr^2 \pi
-  4(x+y)\tan\left(\frac{\pi}{2}(x^2+y^2)\right) 
\end{pmatrix}.
\]

In Figure \ref{fig:stokes} we see the resulting $L^2$- and
$H^1$-errors. Again we observe that $\Upsilon$ becomes the dominant
factor for finer meshes. This result is not covered by the theoretical
findings, however it shows that geometric uncertainty should be taken
into account for the simulations of flow models. 

\begin{figure}[t]
  \begin{center}
    \includegraphics[width=.65\textwidth]{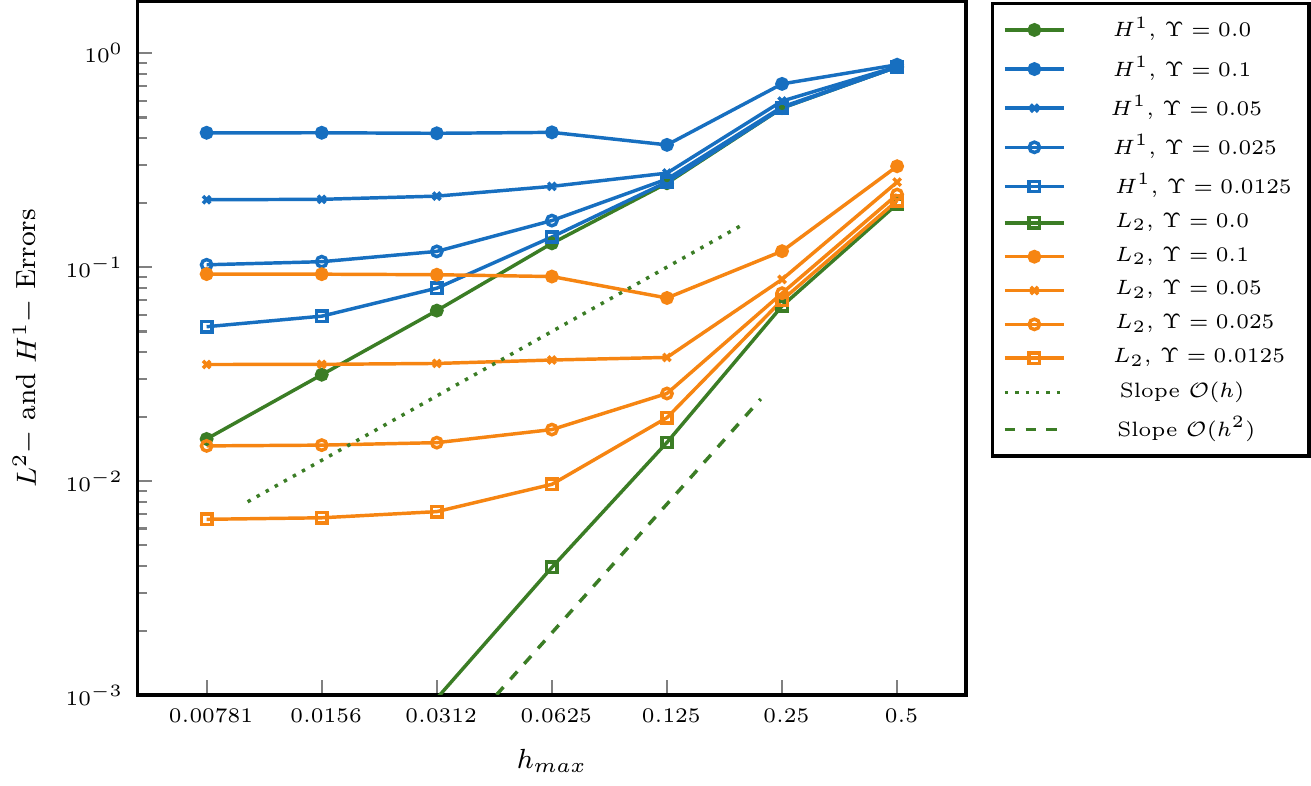}
  \end{center}	
  \caption{$L^2$- and $H^1$-errors w.r.t. mesh-size $h_{max}$ for varying 
  parameter $\Upsilon$ computed for the Stokes problem in two dimensions with 
  linear finite elements.}
  \label{fig:stokes}
\end{figure}

\section{Conclusions}

We have demonstrated that small boundary variations have crucial
impact on the result of the finite element simulations. The
developed error estimates are linear with respect to the maximal
distance $\Upsilon$ between the real and the approximated domains,
cf. Theorem \ref{theorem:main}. We have illustrated the sharp nature of
this bound in the computations performed in
Section~\ref{sec:comp}. 

Particularly, in the case of first and second order approximation we 
observe how the relation between the mesh size $h$ and aforementioned 
$\Upsilon$ impact the resulting $L^2$- and $H^1$-errors. 
The same behavior has been demonstrated numerically for the Stokes system.

In practice we do not have control on the accuracy of the domain
reconstruction. This has shown that it is worth to take into account
the geometric uncertainty when deciding on the mesh-size in order to
avoid unnecessary computational effort.

In this work we have focused on the Laplace problem
\eqref{vessel:laplaceweak}. Additionally, the Stokes system has been
treated numerically and it exhibits similar features. In future work we
will extend this consideration to flow models, in particular the Navier-Stokes equations \cite{MiMiRi2019}. Among the
additional challenges in extending the present work to the
Navier-Stokes system are the consideration of the typical saddle-point
structure of incompressible flow models introducing a pressure
variable~\cite{RannacherRichter2018} and the difficulty of
nonlinearities introduced by the convective term, and thus the
non-uniqueness of solutions~\cite{HeywoodRannacher1990}.  

\begin{acknowledgement}
  Both authors acknowledge the financial support by the Federal Ministry of
  Education and Research of Germany, grant number 05M16NMA as well as
  the GRK 2297 MathCoRe, funded by the Deutsche
  Forschungsgemeinschaft, grant number 314838170. Finally we thank the 
  anonymous reviewers. Their time and effort helped us to significantly improve 
  the manuscript. 
\end{acknowledgement}


\end{document}